\documentclass[12pt]{amsart}
\usepackage{amssymb}

\newtheorem{theorem}{Theorem}[section]

\newtheorem{lemma}{Lemma}[section]

\renewcommand{\v}[1]{\boldsymbol{#1}}

\title[Finding well approximating lattices]{Finding well approximating lattices for a finite set of points}

\author{A. Hajdu}
\address{Faculty of Informatics, University of Debrecen \newline\indent P. O. Box 12, H-4010 Debrecen, Hungary}
\email{hajdu.andras@inf.unideb.hu}

\author{L. Hajdu}
\address{Institute of Mathematics, University of Debrecen \newline\indent P. O. Box 12, H-4010 Debrecen, Hungary}
\email{hajdul@science.unideb.hu}

\author{R. Tijdeman}
\address{Mathematical Institute, Leiden University\newline
\indent Postbus 9512, 2300 RA Leiden, The Netherlands}
\email{tijdeman@math.leidenuniv.nl}

\thanks{Research was supported in part by the OTKA grants K100339, NK101680, K115479 and the projects T\'AMOP-4.2.2.C-11/1/KONV-2012-0001 and VKSZ\_14-1-2015-0072 supported by the European Union, co-financed by the European Social Fund.}



\date{\today}

\begin{document}

\begin{abstract}
In this paper we address the problem of finding well approximating lattices for a given finite set $A$ of points in ${\mathbb R}^n$. More precisely, we search for $\v{o},\v{d_1}, \dots,\v{d_n}\in \mathbb{R}^n$ such that $\v{a}-\v{o}$ is close to $\Lambda=\v{d_1}\mathbb{Z}+\dots+\v{d_n}\mathbb{Z}$ for every $\v{a}\in A$. First we deal with the one-dimensional case, where we show that in a sense the results are almost the best possible. These results easily extend to the multi-dimensional case where the directions of the axes are given, too. Thereafter we treat the general multi-dimensional case. Our method relies on the LLL algorithm. Finally we apply the least squares algorithm to optimize the results. We give several examples to illustrate our approach.
\end{abstract}

\maketitle

\section{Introduction}
\label{intro}

We address the following problem: given a finite set $A$ of points in ${\mathbb R}^n$ which do not fit in a hyperplane, find $\v{o},\v{d_1},\dots,\v{d_n}\in\mathbb{R}^n$ such that the distance of $\v{a}-\v{o}$ to the lattice $\Lambda:=\v{d_1}\mathbb{Z}+ \dots +\v{d_n}\mathbb{Z}$ is relatively small for every $\v{a}\in A$. We shall call $\Lambda$ a {\it well approximating lattice}, ignoring that a shift is made from the origin to $\v{o}$.
As usual, for $U\subset \mathbb{R}^n$, $\v{v}\in \mathbb{R}^n$ and $r\in \mathbb{R}$ we write $U+\v{v}$ for $\{\v{u}+\v{v} : \v{u} \in U\}$ and $rU$ for $\{r\v{u} : \v{u} \in U\}$. Vectors will always be denoted by boldface letters.

Let $A=\{\v{a_1},\dots,\v{a_k}\}\subset {\mathbb R}^n$ be given. If $k \leq n+1$, there is an optimal solution for the problem. In the sequel we assume $k>n+1$. Of course, we can make the distances arbitrarily small by choosing $\v{d_1}, \dots,\v{d_n}$ extremely small. Therefore we need a measure which enables us to compare the quality of solutions in a fair way. To do so we introduce the maximum norm $N_{\Lambda,\v{o}}(A)$ and the square norm $N_{\Lambda,\v{o}}^{(2)}(A)$ by
$$
N_{\Lambda,\v{o}}(A) :=  \max_{\v{a} \in A} \frac{|\v{a}-\v{o}-\Lambda|}{\Delta}\left( {\frac {{\rm diam}~A}{ \Delta}}\right)^{\frac{n}{k-n-1}}
$$
and
$$
N^{(2)}_{\Lambda,\v{o}}(A) := \frac{\sqrt{ \sum_{\v{a}\in A} |\v{a}-\v{o}-\Lambda|^2} }{ \Delta}\left( {\frac {{\rm diam}~A}{ \Delta}}\right)^{\frac{n}{k-n-1}},
$$
where $\Delta$ is the $n$-th root of the lattice determinant of $\Lambda$, and diam$(A)$ is the diameter of $A$.
We set
$$
N(A)=\inf\limits_{\Lambda,\v{o}}N_{\Lambda,\v{o}}(A)\ \ \ \text{and}\ \ \ N^{(2)}(A)= \inf\limits_{\Lambda,\v{o}}N^{(2)}_{\Lambda,\v{o}}(A).
$$
In Section \ref{mainres}, after Theorem \ref{thm2.4} we explain why the above choice of the norms is appropriate in case $n=1$.

We note that the problem in dimension one is close to that of finding approximate greatest common divisors of integers; see the papers \cite{h-g} and \cite{ch}. However, our problem is different, because we approximate by real lattices and, moreover, allow the origin to shift to $\v{o}$. On the other hand, in \cite{h-g} and \cite{ch} algorithms are given to provide all solutions satisfying some condition, whereas we shall provide solutions without claiming completeness or optimality.

In Section \ref{mainres}, we deal with the one-dimensional case. Theorems \ref{thm2.1} and \ref{thm2.2} provide upper bounds for the above norms $N_{\Lambda,\v{o}}(A)$, $N_{\Lambda,\v{o}}^{(2)}(A)$, $N(A)$, $N^{(2)}(A)$. Our main tools are simultaneous Diophantine approximation and (the theory of) the LLL algorithm \cite{lll}. Theorem \ref{thm2.3} yields that the bounds for $N(A)$ are rather sharp. Theorem \ref{thm2.4} shows that homogeneous simultaneous Diophantine approximation appears in a natural way in the study of the problem. Finally, Theorem \ref{thm2.5} tells that if a very well approximating lattice $\Lambda$ exists, the LLL algorithm should find a well approximating lattice.  In Section \ref{1d}, we extend the algorithmic method of Section \ref{mainres} to the multi-dimensional case by applying it to each coordinate axis. In Section \ref{nd} we generalize the method of Section \ref{mainres} to the multi-dimensional case in a simultaneous way. We prove that our strategy provides a good approximation if a very good approximation exists. Finally, in Section \ref{finetuning}, we use the least squares algorithm to optimize the numerical results with respect to the $N^{(2)}(A)$-norm. We illustrate the various methods by examples.

\section{The one-dimensional case}
\label{mainres}

Let $n=1$ and $A=\{a_1,\dots,a_k\}\subset {\mathbb R}$ be given. Then we have, for $o,d$ in ${\mathbb R}$, the maximum norm
$$
N_{d,o}(A)=\frac{ \max\limits_{i=1,\dots,k} |a_i-o-d\mathbb{Z}|}{d} \left( \frac{ \max\limits_{i<j} |a_i-a_j|}{d} \right)^{\frac{1}{k-2}},
$$
and the square norm
$$
N^{(2)}_{d,o}(A) = \frac{ \sqrt{\sum_{i=1}^k |a_i-o-d\mathbb{Z}|^2}}{d} \left( \frac{ \max\limits_{i<j} |a_i-a_j|}{d} \right)^{\frac{1}{k-2}}.
$$
\noindent We give bounds for
$$
N(A)=\inf_{d,o} N_{d,o}(A),\ \ \ N^{(2)}(A)=\inf_{d,o} N^{(2)}_{d,o}(A)
$$
taken over all $d\in {\mathbb R}$ with $0< d \leq {\rm diam} (A)$ and $o \in \mathbb{R}$, and construct pairs $d,o$ which provide good simultaneous approximations. The upper bound on $d$ is to avoid large $d$'s, which would result in the trivial $N(A)=0$.

\begin{theorem}
\label{thm2.1}
For any finite subset $A$ of ${\mathbb R}$ we have $N(A)<1$ and $N^{(2)}(A) < \sqrt{k-2}$.
\end{theorem}

\noindent In fact the proof of Theorem \ref{thm2.1}  implies that there are arbitrarily small $d>0$ such that $N_{d,a_1}(A)< 1.$

It is important to efficiently construct well approximating lattices. In this direction we prove the following result.

\begin{theorem}
\label{thm2.2}
For any finite subset $A$ of ${\mathbb R}$ one can find $d \in {\mathbb R_{>0}}$ with $N_{d,a_1}(A)<2^{(k-1)/4}$ and $N^{(2)}_{d,a_1}(A) < 2^{(k-1)/4}\sqrt{k-2}$ in polynomial time.
\end{theorem}

Theorem \ref{thm2.1} is in some sense the best possible:

\begin{theorem}
\label{thm2.3}
There exist arbitrarily large finite subsets $A$ of $\mathbb{R}$ such that $N(A)>c_1$, where $c_1$ is a positive number depending only on $A$.
\end{theorem}

The next result shows that existence of a very good inhomogeneous approximating lattice implies the existence of a quite good homogeneous simultaneous diophantine approximation.

\begin{theorem}
\label{thm2.4}
Let $A=\{a_1,\dots,a_k\}$ be a set of real numbers with $k>2$ and $a_1<\dots <a_k$. Suppose that for some $d,o\in {\mathbb R}$ with $0<d \leq a_k-a_1$ we have $N_{d,o}<c_2$ where $c_2$ is a positive constant. Then there exists a positive integer $q$ such that
$$
\max\limits_{1\leq i\leq k} ||q\alpha_i||< 6c_2 q^{-1/(k-2)}\ \ \ \text{and}\ \ \ \left\vert \frac{a_k-a_1}{d}-q\right\vert<3c_2 q^{-1/(k-2)},
$$
where $\alpha_i=(a_i-a_1)/(a_k-a_1)$ $(i=2,\dots,k-1)$, and $||.||$ denotes the distance to the nearest integer.
\end{theorem}

With these theorems it is possible to explain why $N(A)$ is a fair norm, at least in dimension one. In the first place we should compensate for scaling. If we multiply all $a_i$ by a number $\alpha >0$, then all the distance are also multiplied by $\alpha$ and so, by dividing by $d$ we neutralize scaling. We further have to compensate for the value of $d$, which is close to $(a_k-a_1)/q$. According to Theorem \ref{thm2.4} the expected distances from $q\alpha_i$ to the nearest lattice point is of the order $q^{-1/(k-2)}$. Therefore we multiply by $((a_k-a_1)/d)^{1/(k-2)}$ for compensation.
\vskip.2cm

\noindent {\bf Remark.} Since $\max_{i<j} (a_j-a_i)/d$ is close to $q$, the expected value of $N_{d,o}(A)$ for a random $q$ is not bounded, but $(0.5-o_k(1))q^{1/(k-2)}$.
\vskip.3cm

We shall present some examples using the LLL procedure of Maple 15. Here we have to choose some parameters. Without loss of generality we may assume $a_1 < a_2 < \dots < a_k$. First we subtract $a_1$ from every element of $A$. Then we scale all the obtained numbers by dividing them by the largest distance between the numbers $a_i-a_1$, i.e. by $a_k-a_1$. Further we choose an $\varepsilon$ of order $d_0^{1/(k-2)}$ where $d_0$ is the size of the desired $d$. It may be good to try various values of $\varepsilon$ and to compare the different outcomes.

\vskip.4cm

\noindent{\bf Example 2.1.} We choose $k=6$ and
$$
A=\{a_1=0.814258,\ a_2=1.294837,\ a_3=2.237840,
$$
$$
a_4=2.764132, a_5=4.295116,\ a_6=7.733842\}.
$$
By subtracting $a_1$ and dividing by $a_6-a_1= \max_{i<j} |a_j-a_i|$, computing in 10-digit precision, but writing in 4-digit precision, we get the normalized numbers
$$
(a^N_1, a^N_2, a^N_3, a^N_4, a^N_5, a^N_6) = (0, 0.0695, 0.2057, 0.2818, 0.5030, 1).
$$
We apply the LLL algorithm\footnote{Here and elsewhere, we always use the implementation of LLL in Maple 15.} with $\varepsilon = 10^{-3}$ to the matrix
$$
T=
\begin{pmatrix}
1&0&0&0&0\\
0&1&0&0&0\\
0&0&1&0&0\\
0&0&0&1&0\\
0.0695 & 0.2057 & 0.2818 & 0.5030 & 0.0010
\end{pmatrix}
$$
to get the matrix
$$
B=
\begin{pmatrix}
0.0277&0.1197&0.0549&-0.0426&-0.0140\\
0.0982&-0.0491&-0.0854&-0.1012&0.1310\\
0.1514&-0.0605&-0.1313&-0.0632&-0.1850\\
0.2778&-0.1771&0.1272&0.0122&0.0040\\
-0.1118&-0.1044&0.20242&-0.2679&-0.0880
\end{pmatrix}
.
$$
Subsequently we compute the integer matrix
$$
S:= BT^{-1}=
\begin{pmatrix}
1&3&4&7&-14\\
-9&-27&-37&-66&131\\
13&38&52&93&-185\\
0&-1&-1&-2&4\\
6&18&25&44&-88
\end{pmatrix}
.
$$
We conclude that $(p_1,p_2,\dots,p_6)=(0,1,3,4,7,14)$ with $q=14$ yields a good approximation. The second row of $S$ gives other good values: $(p'_1,p'_2,\dots,p'_6)=(0,9,27,37,66,131)$ with $q'=131$. We put $o=a_1$ and choose
$$
d=(a_6-a_1)/q=0.4943~{\rm and}~d'=(a_6-a_1)/q'=0.0528,
$$
as lattice basis, respectively. Thus the points $(a_1, \dots, a_6)$ are close to $o+d\cdot(0,1,3,4,7,14)$ and to $o+d'\cdot(0,9,27,37,66,131)$. The errors of approximation of $A$ are given by
$$
N_{d,o}(A)=\frac{0.0592}{0.4943} \sqrt[4]{\frac {6.9196}{0.4943}}=0.2316
$$
and
$$
N_{d',o}(A)=\frac{0.0053}{0.0528} \sqrt[4]{\frac {6.9196}{0.0528}}=0.3423.
$$
Hence
$$
N(A)\leq \min(0.2316,0.3423)=0.2316
$$
which is less than
$1$ and $2^{(k-1)/4}=2.3784$,
the bounds of Theorems \ref{thm2.1} and \ref{thm2.2}, respectively. In a similar way we find
$$
N_{d,o}^{(2)}(A) = 0.2731,\ \ \ N_{d',o}^{(2)}(A) = 0.5819.
$$

\vskip.2truecm

If we start with $\varepsilon = 10^{-2}$ in place of $\varepsilon = 10^{-3}$ and we follow the same procedure, then we obtain again $q=14$, yielding the same error values.

\vskip.3truecm

\noindent{\bf Example 2.2.} In the second example we choose
$$
A=\{a_1=0,\ a_2=\sqrt{3}=1.7321,\ a_3=\sqrt{5}=2.2361,
$$
$$
a_4=\sqrt{7}=2.6458,\ a_5=\sqrt{11}=3.3166,\ a_6=\sqrt{13}= 3.6056\}.
$$
The largest distance is $a_6-a_1=\sqrt{13}=3.6056$. Dividing by this number gives the normalized numbers
$$
(a^N_1, a^N_2, a^N_3, a^N_4, a^N_5, a^N_6) =
(0,0.4804,0.6202, 0.7338,0.9199,1).
$$
Applying the LLL algorithm to the matrix
$$
\begin{pmatrix}
1&0&0&0&0\\
0&1&0&0&0\\
0&0&1&0&0\\
0&0&0&1&0\\
0.4804&0.6202&0.7338&0.9199&0.0010
\end{pmatrix}
$$
we get a matrix $B$ with first row
$$
\begin{pmatrix}
-0.0577&-0.0261&-0.0699&0.0201&-0.1500
\end{pmatrix}
.
$$
Subsequently we compute the integer matrix $S= BT^{-1}$ and obtain
$$
\begin{pmatrix}
72&93&110&138&-150
\end{pmatrix}
$$
as its first row. We conclude that $q=150$ together with
$$
(p_1,p_2,\dots,p_6)=(0,72,93,110,138,150)
$$
yields a good approximation. We get
$$
o=a_1=0,d=\sqrt{13}/q=0.0240.
$$
Thus the points $(a_1, \dots, a_6)$ are close to $d\cdot(0,72,93,110,138,150)$. In this way we obtain
$$
N(A)\leq N_{d,o}(A)=0.2447,\ \ \ N^{(2)}(A) \leq N^2_{d,o}(A) =0.3374.
$$
The upper bounds are again less than the  bounds from Theorems \ref{thm2.1} and \ref{thm2.2}.
\vskip.2cm

\noindent If we start with $\varepsilon = 10^{-2}$ in place of $\varepsilon = 10^{-3}$ and we follow the same procedure, we obtain $q=8$, $N_{d,o}(A)= 0.6036, N_{d,o}^{(2)}(A)= 0.6970$. \qed
\vskip.3cm

The next theorem shows that a well approximating lattice with similarly sized $d$ cannot be much better than the lattice we found by the LLL algorithm.

\begin{theorem} \label{thm2.5}
Suppose the LLL algorithm applied to $a_1=0< a_2<\dots<a_{k-1} < a_k=1$ and $\varepsilon > 0$ yields $\v{b_1}$ as the first row of the matrix $B$. Then for every $d'> \frac{\sqrt{k}~ 2^{k/2-1}\varepsilon}{|\v{b_1}|}$ and every choice of integers $p'_1=0,p'_2, \dots, p'_k=q'$ we have
$$ \max_{i=1, \dots, k} |a_i-p'_id'| \geq \frac {2^{1-k/2}}{q'\sqrt{k}} |\v{b_1}|,$$
whereas the LLL algorithm finds integers $p_1=0, p_2, \dots, p_{k-1}, \\p_k=q=1/d$ such that $$ \max_{i=1, \dots, k} |a_i-p_id| \leq \frac{|\v{b_1}|}{q}.$$
\end{theorem}

For the proofs of the theorems we first turn to Theorem \ref{thm2.2}.

\begin{proof}[Proof of Theorem \ref{thm2.2}] Let $A= \{a_1, \dots, a_k\}\subset \mathbb{R} \ \ (k \geq 3)$.
Let $o=a_1$, and $\alpha_i=(a_i-a_1)/(a_k-a_1)$ for $i=2,\dots,k-1$. Then, by Lemma 8 of \cite{bs}, for any $\varepsilon\in (0,1)$ integers $p_2,\dots,p_{k-1}$ and $q$ can be found in polynomial time such that
$$
|p_i-q\alpha_i|\leq\varepsilon\ \text{for}\ i=2,\dots,k-1,
$$
$$
1\leq q\leq 2^{(k-2)(k-1)/4}\varepsilon^{-k+2}.
$$
We note that the same assertion follows already from Proposition 1.39 of \cite{lll} in case of $A\subset \mathbb{Q}$. Put
$d=(a_k-a_1)/q$.
For $i=2,\dots,k-1$ we have
\begin{equation} \label{ub}
|p_i-q\alpha_i|\leq\varepsilon\leq 2^{(k-1)/4}q^{-1/(k-2)}.
\end{equation}
Since
\begin{equation} \label{eq}
|p_i-q\alpha_i|=|p_id-(a_i-o)|/d\ \ \ (i=2,\dots,k-1),
\end{equation}
we obtain by \eqref{ub} and \eqref{eq}, putting $p_1=0, p_k=q$,
$$
N(A)\leq N_{d,o}(A)= \frac{\max\limits_{i=1,\dots,k} |p_id-(a_i-o)|}{d} \left(\frac{a_k-a_1}{d}\right)^{\frac{1}{k-2}}
\leq 2^{(k-1)/4}.
$$
For the upper estimate for $N^{(2)}(A)$ we use the bound 0 for $i=1,k$ and the bound $2^{(k-1)/4}$ for $i=2, \dots, k-1$.
\end{proof}

\begin{proof}[Proof of Theorem \ref{thm2.1}] The proof goes along the same lines as that of Theorem \ref{thm2.2}. The only difference is that in place of Lemma 8 of \cite{bs}, we use a theorem of Dirichlet (see Schmidt \cite{sch}, Chapter II) in \eqref{ub}, guaranteeing the existence of an integer $q$ such that
$$
||q\alpha_i||<q^{-1/(k-2)}\ \ \ (i=2,\dots,k-1)
$$
where $||.||$ denotes the distance to the nearest integer.
\end{proof}

\begin{proof}[Proof of Theorem \ref{thm2.3}]
Let $a_1=0,a_2,\dots,a_{k-1}, a_{k}=1$ be $k$ numbers in a real algebraic number field of degree $k-1$ such that $a_2,\dots,a_{k-1},1$ are linearly independent over ${\mathbb Q}$ and $a_1<a_2<\dots < a_{k-1} <a_k$. Then, by Theorem III on p. 79 of \cite{ca}, there exists a positive real number $\gamma$ (depending only on $a_2,\dots,a_{k-1}$) such that for all positive integers $q$ we have
\begin{equation} \label{cas}
\max\limits_{2\leq i\leq k-1}||q a_i||\geq \gamma q^{-1/(k-2)}.
\end{equation}
Suppose that for some real numbers $o,d, \varepsilon$ with $0<d \leq 1, 0 < \varepsilon < 1/4 $ we have
\begin{equation}
\label{eqapprox3}
\max\limits_{1\leq i\leq k} |a_i-p_id-o|<\varepsilon d,
\end{equation}
where the $p_i$ are integers minimizing the expressions $|a_i-p_id-o|$.
Then
\begin{equation}
\label{eqapprox0}
\max\limits_{1\leq i\leq k} |(a_i-a_1)-(p_i-p_1)d|<2\varepsilon d.
\end{equation}
Set $t=1/d$ and let $||.||$ denote the distance to the nearest integer. Then we have, by $a_1=0, a_k=1$,
\begin{equation}
\label{eqapprox1}
\max\limits_{2\leq i\leq k-1}||ta_i|| <2\varepsilon \ \ \ \text{and}\ \ \ ||t||<2\varepsilon.
\end{equation}
Let $q$ be an integer with $|t-q| \leq 1/2.$
Then for $i=2,\dots,k-1$ we obtain, by \eqref{cas} and \eqref{eqapprox1},
$$
||ta_i||\geq ||qa_i||-|t-q|a_i = ||qa_i||-||t||a_i\geq
\gamma q^{-1/(k-2)}-2\varepsilon.
$$
Observe that $t/q\geq 1/2$. By choosing
$\varepsilon= \gamma /(4q^{1/(k-2)}),$
the above inequality yields
$$
\max\limits_{2\leq i < k}||ta_i||\geq 2 \varepsilon
$$
which contradicts \eqref{eqapprox1}. Hence for the chosen value of $\varepsilon$ inequality \eqref{eqapprox3} does not hold. It follows that
$$N_{d,o}(A)= \frac{ \max\limits_{i=1,\dots,k} |a_i-o-d\mathbb{Z}|}{d^{1+1/(k-2)}} \geq \varepsilon t^{1/(k-2)} \geq \frac{\gamma}{4}\left(\frac tq \right)^{1/(k-2)} \geq \frac{\gamma}{8}.$$
\end{proof}

\begin{proof}[Proof of Theorem \ref{thm2.4}] Let $o$ and $d$ be real numbers with $0<d \leq a_k-a_1$. Suppose that $N_{d,o}<c_2$ where $c_2$ is an arbitrary positive real. Then we have
$$
d^{-1}\max\limits_{1\leq i\leq k} |a_i-p_id-o|<\varepsilon:= c_2 \left( \frac{d}{a_k-a_1} \right)^{1/(k-2)}.
$$
Hence $$d^{-1}\max\limits_{1\leq i\leq k} |a_i-a_1-(p_i-p_1)d|<2\varepsilon.$$
Put $t=(a_k-a_1)/d$ and $\alpha_i=(a_i-a_1)/(a_k-a_1)$ for $i=1, \dots, k$. Then
\begin{equation}
\label{eqapprox2}
\max\limits_{2\leq i\leq k-1}||t\alpha_i|| <2\varepsilon \ \ \ \text{and}\ \ \ ||t||<2\varepsilon.
\end{equation}
Thus choosing integers $q$ such that $|t-q|=||t||$ and $p_i$ such that $|t\alpha_i - p_i| = ||t \alpha_i||$, by \eqref{eqapprox2} we get for all $i=2,\dots,k-1$
$$
|q\alpha_i-p_i|<|t\alpha_i-p_i|+|t-q|\alpha_i< 4\varepsilon.
$$
Hence by the definition of $\varepsilon$ and $t$,
noting that by $t \geq 1$ we have $q>0$ and $q/t \leq 4/3$, we derive
$$
\max\limits_{2\leq i\leq k-1} ||q\alpha_i||<4 \varepsilon =4c_2 t^{-1/(k-2)}=4c_2\frac{(q/t)^{1/(k-2)}}{q^{1/(k-2)}}
<\frac{16}{3}c_2 q^{-1/(k-2)}.
$$
Similarly, by \eqref{eqapprox2},
$$
\left\vert \frac{a_k-a_1}{d}-q\right\vert=|t-q|<2\varepsilon <\frac{8}{3}c_2 q^{-1/(k-2)}.
$$
\end{proof}

\begin{proof}[Proof of Theorem \ref{thm2.5}]
Let $L$ be the lattice generated by the $k-1$ vectors
$$
(1,0,0,\dots,0),\dots,(0,\dots,0,1,0), (a_2,\dots,a_{k-1},\varepsilon)
$$
in ${\mathbb R}^{k-1}$. According to \cite{lll} Theorem 1.11 we have, for every lattice point $\v{x}\in L$,
$$ |\v{b_1}|^2 \leq 2^{k-2} |\v{x}|^2,$$
where $\v{b_1}$ is the the shortest vector of an LLL-reduced basis of $L$. Clearly, any lattice point $\v{x}\in L$ can be written as
$$\v{x} = \left(p'_2 - q'a_2, \dots, p'_{k-1} - q'a_{k-1}, -q' \varepsilon \right),$$
where $p'_2, \dots, p'_{k-1}, p'_k=q'$ are integers.
Hence $$ \max \left( \max_{i=2, \dots, k-1} |p'_i-q'a_i|, |q'\varepsilon| \right)^2 \geq \frac {2^{2-k}}{k} |\v{b_1}|^2.$$
So provided that $|q' \varepsilon| < \frac {2^{1-k/2}}{\sqrt{k}} |\v{b_1}|$, we have
$$ \max_{i=2, \dots, k-1} |p'_i-q'a_i| \geq \frac {2^{1-k/2}}{\sqrt{k}}~ |\v{b_1}|.$$
Put $d' = 1/q'$. Then, provided that $d' > \frac {2^{k/2-1}\sqrt{k}~ \varepsilon}{|\v{b_1}|}$, we have
$$ \max_{i=2, \dots, k-1} |a_i-p'_id'| \geq \frac {2^{1-k/2}}{q'\sqrt{k}} |\v{b_1}|.$$

For the second part, observe that it follows from the algorithm in \cite{lll} that for some integers $p_2,\dots,p_{k-1},q$ we have
$$  \sum_{i=2}^{k-1} (p_i-qa_i)^2 +(q \varepsilon)^2 = |\v{b_1}|^2.$$
Hence $|p_i-qa_i| \leq |\v{b_1}|$ which implies $|a_i-p_i/q| \leq |\v{b_1}|/q$ for $i=1, \dots, k$ in view of $a_1=0=p_1, a_k=1, q=p_k$.
\end{proof}

\section{Lattices with the basis vectors in given directions}
\label{1d}

For any given finite set $A \subset \mathbb{R}^n$ which do not fit into a hyperplane we generate vectors $\v{d}=(d_1, \dots, d_n)$, $\v{o}=(o_1, \dots, o_n) \in \mathbb{R}^n$ such that every element of $A-\v{o}$ is close to the lattice $\Lambda = (d_1 \mathbb{Z},\dots,d_n\mathbb{Z})$. That is, we approximate $A$ with a rectangular lattice. By a linear transformation one can transfer any other prescribed set of lattice basis vectors to this case.

Let $A=\{\v{a_1},\dots,\v{a_k}\}$ with $\v{a_i}=(a_{i1}, \dots, a_{in}) \in \mathbb{R}^n$ for $i=1, \dots, k$. As norms we use
$$
N_{\v{d},\v{o}}(A) := \frac{\max_{\v{a} \in A} |\v{a}-\v{o}-\Lambda| }{ \Delta}\left( {\frac {{\rm diam}~A}{ \Delta}}\right)^{\frac{n}{k-n-1}},
$$
and
$$
N^{(2)}_{\v{d},\v{o}}(A) :=  \frac{\sqrt{ \sum_{\v{a} \in A} |\v{a}-\v{o}-\Lambda|^2} }{ \Delta}\left( {\frac {{\rm diam}~A}{ \Delta}}\right)^{\frac{n}{k-n-1}},
$$
where $\Delta := \left(\prod_{i=1}^n |d_i| \right)^{1/n}$ is the $n$-th root of the lattice determinant. Again we fix $\v{o}=\v{a_1}$.

\vskip.2cm

We illustrate by two examples how the results of Section \ref{mainres} can be used in this case.

\vskip.2truecm

\noindent{\bf Example 3.1.} We combine Examples 2.1 and 2.2. We choose $n=2$, $k=6$ and
$$
A=\{\v{a_1}=(0.814258,0), \ \v{a_2}=(1.294837,\sqrt{3}),\ \v{a_3}=(2.237840,\sqrt{5}),
$$
$$
\v{a_4}=(2.764132,\sqrt{7}),\ \v{a_5}=(4.295116,\sqrt{11}), \ \v{a_6}=(7.733842,\sqrt{13})\}.
$$
We again calculate in 10-digit precision, write in 4-digit precision and put $\varepsilon  = 10^{-3}$. Recall
$$
\sqrt{3}=1.7321,\sqrt{5}=2.2361,
\sqrt{7}=2.6458,\sqrt{11}=3.3166,\sqrt{13}= 3.6056.
$$
From Examples 2.1 and 2.2 we obtain
$\v{d}=(0.4943, 0.0240)$ and the approximating points become
$$
(0.8143,0.0000), \ (1.3085,1.7307),\ (2.2970, 2.2354),
$$
$$
(2.7913,2.6441), \ (4.2741,3.3171), \ (7.7334, 3.6056).
$$
Using $\Delta=\sqrt{0.4943\cdot 0.0240}=0.1089$ and diam$(A)= |\v{a_6} - \v{a_1}| = 7.8026$,
the values $N_{\v{d},\v{o}}(A) = 9.3622$, $N^{(2)}_{\v{d},\v{o}}(A)= 11.0453$ follow.\\
\qed
\vskip.3cm

In Example 3.1 both coordinates are in increasing order. Of course, this need not be the case. In the next example we permute the second coordinates.

\vskip.2cm

\noindent{\bf Example 3.2.} We choose $n=2$, $k=6$ and
$$
A=\{\v{a_1}=(0.814258,\sqrt{5}),\ \v{a_2}=(1.294837,0),\ \v{a_3}=(2.237840,\sqrt{13}),
$$
$$
\v{a_4}=(2.764132,\sqrt{3}),\ \v{a_5}=(4.295116,\sqrt{11}),\ \v{a_6}=(7.733842,\sqrt{7})\}.
$$
We again calculate in 10-digit precision and write in 4-digit precision, and take $\varepsilon=10^{-3}$. From Examples 2.1 and 2.2 we obtain $\v{d}=(0.4943, 0.0240)$. The approximating lattice remains the same. Therefore the approximating points are obtained from Example 3.1 by making the corresponding permutation. Using $\Delta =  0.1089$ and diam$(A) = |\v{a_6} - \v{a_2}| = 6.9614$, we obtain the values $N(A)\leq N_{\v{d},\v{o}}(A) =8.6761$, $N^{(2)}(A) \leq N^{(2)}_{\v{d},\v{o}}(A)= 10.2364$. That the values are smaller than the corresponding values in Example 3.1 is mainly due to the smaller diameter.
\qed

\section{Approximating with general lattices}
\label{nd}

In this section we present a method for finding well approximating general lattices. First we prove that our strategy provides a good approximation if a very good approximation exists. We generalize the method of Section \ref{mainres} and illustrate how it works through some examples.

\begin{theorem}
\label{apprth}
If $A$ is a finite set of $k$ points in $\mathbb{R}^n$ with $k>n$ such that they do not fit into an $n-1$-dimensional linear manifold and there exist a lattice $\Lambda$, a point $\v{o} \in \mathbb{R}^n$ and an $\varepsilon >0$ such that $|\v{a} - \v{o} -\Lambda| < \varepsilon<1/2$ for all $\v{a} \in A$, then there exist an affine (inhomogeneous linear) transformation $V: \mathbb{R}^n \to \mathbb{R}^n $ and $\v{a_1}, \dots, \v{a_{n+1}} \in A$ such that $V$ maps the lattice point in $\Lambda$ nearest to $\v{a_i}- \v{o}$ to $\v{a_i} - \v{o}$ itself for $i=1,\dots, n+1$ and $|\v{a} -\v{o}- V\Lambda | < 2^n \varepsilon$ for all $\v{a} \in A$.
\end{theorem}

\noindent
We shall use the following lemma.

\begin{lemma} \label{boxlem}
Let a rectangular box
$$
B= \{ 0 \leq x_1 \leq b_1, |x_2| \leq b_2, \dots, |x_n| \leq b_n: b_1, b_2, \dots, b_n \in \mathbb{R}_{> 0}\}
$$
be given in $\mathbb{R}^n$. Set $\v{a''_1}=(b_1, 0, \dots,0)$ and for $i=2,\dots,n$ let $\v{a''_i}$ be a point of the form $(b_{1i}, \dots, b_{ii},0, \dots,0)$ with
\begin{equation}
\label{boxB}
0 \leq b_{1i} \leq b_1, |b_{ji}| \leq b_j ~ {\rm for}~j=2, \dots i-1, b_{ii}=b_i.
\end{equation}
Then every point $\v{x}\in B$ can be written as $\lambda_1 \v{a''_1}+\dots+\lambda_n\v{a''_n}$ with $|\lambda_i| \leq 2^{n-i}$ for $i=1,\dots,n$.
\end{lemma}

\begin{proof}
By induction on $n$. For $n=1$ the assertion is obvious.
Suppose the statement is true for $n$. Then set
$$
B_{n+1} =  \{ 0 \leq x_1 \leq b_1,  |x_2| \leq b_2, \dots, |x_{n+1}| \leq b_{n+1}: b_1,\dots, b_{n+1} \in \mathbb{R}_{> 0}\}.
$$
We identify the tuple $(x_1, \dots, x_n)\in B_n$ with $(x_1, \dots, x_n,0)\in B_{n+1}$ and write $\v{b_{n+1}} = (0,\dots, 0, b_{n+1})$. Any point $\v{x}\in B_{n+1}$ can be written as $\v{x'}+\mu\v{b_{n+1}}$ with $\v{x'}\in B_n$ and $-1 \leq \mu \leq 1$. By the induction hypothesis $\v{x'}$ can be written as $\lambda _1 \v{a''_1} + \dots + \lambda_n\v{a''_n}$ with $|\lambda_i| \leq 2^{n-i}$ for $i=1,\dots,n$. By the same hypothesis $\v{a''_{n+1}}$ can be written as $\mu \v{a''_1} + \dots +\mu_n \v{a''_n}+\v{b_{n+1}}$ with $|\mu_i| \leq 2^{n-i}$ for $i=1, \dots, n$. Hence
$$
\v{x} = \lambda_1 \v{a''_1} + \dots + \lambda_n \v{a''_n} + \mu (\v{a''_{n+1}} - \mu \v{a''_1} \dots - \mu_n \v{a''_n})
$$
$$
= (\lambda_1 - \mu \mu_1) \v{a''_1} + \dots +(\lambda_n - \mu \mu_n) \v{a''_n} + \mu \v{a''_{n+1}}.
$$
It follows that for $i=1, \dots, n$ the coefficient of $\v{a''_i}$ satisfies
$$
|\lambda_i - \mu \mu_i| \leq 2^{n-i} + 2^{n-i}= 2^{n+1-i}.
$$
\end{proof}

\begin{proof} [Proof of Theorem \ref{apprth}]
Suppose that a lattice $\Lambda \subset \mathbb{R}^n$, $\v{o} \in \mathbb{R}^n$ and $\varepsilon>0$ are such that $|\v{a} - \v{o} -\Lambda| < \varepsilon<1/2$ for all $\v{a} \in A$.
Let $\v{a'}$ denote the lattice point of $\Lambda$ nearest to $\v{a}-\v{o}$ for $\v{a}\in A$ and let $A'$ be the set of such lattice points $\v{a'}$. Then, for $\v{a}\in A$,
$|\v{a} -\v{o}-\v{a'}| < \varepsilon$. Without loss of generality we assume that $\v{a'_1}, \v{a'_2}$ are such that $|\v{a'_2} - \v{a'_1}|$ equals the diameter $b_1$ of $A'$, $\v{a'_3} \in A'$ is such that the distance $b_2$ from $\v{a'_3}$ to the line through $\v{a'_1}$ and $\v{a'_2}$ is maximal, $\v{a'_4}$ is such that the distance $b_3$ from $\v{a'_4}$ to the plane through $\v{a'_1}, \v{a'_2}, \v{a'_3}$ is maximal, and so on up to $b_n$. Let $\v{a'_{n+2}},\dots, \v{a'_k}$ be the remaining points of $A'$. Label the elements of $A$ accordingly. Hence, for $i=1, \dots,k$ we have $|\v{a_i} -\v{o}-\v{a'_i}| < \varepsilon$.
Set
$$
B= \{ 0 \leq x_1 \leq b_1, |x_2| \leq b_2, \dots, |x_n| \leq b_n\}.
$$
Consider the affine transformation $U$ for which $U\v{a'_1} =(0,\dots,0)$ and $U\v{a'_{i+1}}=(b_{1i},b_{2i},\dots, b_{i-1,i},b_i,0\dots,0)$ $(i=1,\dots,n)$. Then by the choice of $\v{a'_1},\dots,\v{a'_{n+1}}$ we have that $U\v{a'_i}\in B$ for $i=1,\dots,k$, and also that $U\v{a_1},\dots,U\v{a_{n+1}}$ satisfy the conditions of Lemma \ref{boxlem}. Hence every $U\v{a'_i}$ can be written as $q_{1i}U \v{a'_2} +q_{2i} U\v{a'_3}+\dots +q_{ni} U\v{a'_{n+1}}$ with $|q_{ji}| \leq 2^{n-j}$ for $j=1, \dots, n$. Therefore every $\v{a'_i}$ can be written as $\v{a'_i}= q_{1i}\v{a'_2} +q_{2i} \v{a'_3}+ \dots +q_{ni} \v{a'_{n+1}}$ with $|q_{ji}| \leq 2^{n-j}$ for $j=1, \dots, n$.

Let $V$ be the affine transformation which maps $\v{a'_i}$ to $\v{a_i}-\v{o}$ for $i=1,\dots,n+1$. Then $|\v{a'_i} - V\v{a'_i}| < \varepsilon$ for $i=1,\dots, n+1$.
Let $\Lambda' = V \Lambda$. Then for $i=1, \dots, k$ we have $V\v{a'_i} \in \Lambda'$ and
$$
|\v{a_i} - \v{o}-V\v{a'_i}| \leq |\v{a_i}-\v{o} - \v{a'_i}| + |\v{a'_i} -V\v{a'_i}|
$$
$$
\leq \varepsilon +\sum_{j=1}^n |q_{ji}| |\v{a'_{j+1}} -V\v{a'_{j+1}}| \leq \varepsilon +\varepsilon \sum_{j=1}^n 2^{n-j}  = 2^n \varepsilon.
$$
\end{proof}

Theorem \ref{apprth} shows that the original inhomogeneous problem is not far from a homogeneous problem. We follow this idea in our treatment.

\vskip.1cm

Let $A=\{\v{a_1},\dots,\v{a_k}\}\subset {\mathbb R}^n$, and write $\v{a_i}=(a_{1i},\dots,a_{ni})$ for $i=1,\dots,k$.
we apply first a normalization as we did in Section \ref{mainres}, too. We choose $n+1$ points of $A$, say $\v{a_1}, \dots, \v{a_{n+1}}$ as in the proof of Theorem \ref{apprth}. We first apply a shift which moves $\v{a_1}$ to the origin $\v{a_1^N}=\v{o}$ and subsequently a linear transformation $W$ which moves $\v{a_{i+1}}-\v{a_1}$ to the unit vector $\v{a_{i+1}^N}:=\v{e_i}$ for $i=1,\dots,n$. By this shift and transformation the $k-n-1$ remaining points of $A$, $\v{a_{n+2}},\dots,\v{a_{k}}$ say, move to points $\v{a_{n+2}^N},\dots,\v{a_k^N}$, respectively. By doing so we save $n+1$ columns in the application of the LLL algorithm.

Consider the $(k-1)\times (k-1)$ matrix
$$
T:=
\begin{pmatrix}
1&0&\cdots&0&0&0&\cdots&0\\
0&1&\cdots&0&0&0&\cdots&0\\
\vdots&\vdots&\vdots
&\vdots&\vdots&\vdots&\vdots&\vdots\\
0&0&\cdots&1&0&0&\cdots&0\\
a^N_{1(n+2)}&a^N_{1(n+3)}&\cdots&
a^N_{1k}&\varepsilon&0&\cdots&0\\
a^N_{2(n+2)}&a^N_{2(n+3)}&\cdots&
a^N_{2k}&0&\varepsilon&\cdots&0\\
\vdots&\vdots&\vdots&\vdots&\vdots
&\vdots&\vdots&\vdots\\
a^N_{n(n+2)}&a^N_{n(n+3)}&\cdots&
a^N_{nk}&0&0&\hdots&\varepsilon
\end{pmatrix}
$$
with a small $\varepsilon$, to be specified later in an appropriate way. Let $B=(b_{ij})_{i,j=1,\dots,k-1}$ denote the $(k-1)\times (k-1)$ matrix obtained by applying the LLL algorithm to the rows of $T$ as in \cite{lll}. We expect the entries $b_{ij}$ to be relatively small. There exists a unimodular $(k-1)\times (k-1)$ transformation matrix $S$ such that $B=ST$ holds. We set
$$
S:=
\begin{pmatrix}
p_{11}&\cdots&p_{1(k-n-1)}&q_{11}&\cdots&q_{1n}\\
p_{21}&\cdots&p_{2(k-n-1)}&q_{21}&\cdots&q_{2n}\\
\vdots&\vdots&\vdots&\vdots&\vdots&\vdots\\
p_{(k-1)1}&\cdots&p_{(k-1)(k-n-1)}&q_{(k-1)1}&\cdots&q_{(k-1)n}\\
\end{pmatrix}
.
$$
Note that all the $p_{ij}$ $(i=1,\dots,k-1;j=1,\dots,k-n-1)$ and the $q_{ij}$ $(i=1,\dots,k-1;j=1,\dots,n)$ are integers. For $i=1,\dots,k-1$ and $j=1,\dots,k-n-1$ we have
$$
b_{ij}=p_{ij} + q_{i1}a_{1(n+1+j)}^N+q_{i2}a_{2(n+1+j)}^N+\dots+q_{in}a_{n(n+1+j)}^N.
$$
Since the $b_{ij}$ are `small', it means that (for the above choice of the indices $i,j$)
$$
-p_{ij} \approx q_{i1}a^N_{1(n+1+j)}+q_{i2}a^N_{2(n+1+j)}+\dots+q_{in}a^N_{n(n+1+j)}.
$$
Take indices $i_1,\dots,i_n$ with $1\leq i_1<\dots<i_n\leq k-1$ such that
$$
Q:=
\begin{pmatrix}
q_{i_11}&\hdots&q_{i_1n}\\
\vdots&\vdots&\vdots\\
q_{i_n1}&\hdots&q_{i_nn}
\end{pmatrix}
.
$$
is invertible. Then, recalling $\v{a_j^N}=W(\v{a_j}-\v{o})$ for all $j=1,\dots,k$, we find that
$$
-W^{-1}Q^{-1}
\begin{pmatrix}
-q_{i_1(j-1)}\\
\vdots\\
-q_{i_n(j-1)}
\end{pmatrix}
+
\begin{pmatrix}
o_1\\
\vdots\\
o_n
\end{pmatrix}
=
\begin{pmatrix}
a_{1j}\\
\vdots\\
a_{nj}
\end{pmatrix}
$$
for $j=2,\dots,n+1$, and
$$
-W^{-1}Q^{-1}
\begin{pmatrix}
p_{1j}\\
\vdots\\
p_{nj}
\end{pmatrix}
+
\begin{pmatrix}
o_1\\
\vdots\\
o_n
\end{pmatrix}
\approx
\begin{pmatrix}
a_{1j}\\
\vdots\\
a_{nj}
\end{pmatrix}
$$
for $j=n+2,\dots,k$.
This means that writing $\v{d_i}$ for the $i$-th column of $-W^{-1}Q^{-1}$, we get that the point $\v{a_j}$ is just the shifted lattice point
$$
\v{o}-q_{i_1(j-1)}\v{d_1}-q_{i_2(j-1)}\v{d_2}-\dots-q_{i_n(j-1)}\v{d_n}
$$
for $j=2,\dots,n+1$, and it is close to
$$
\v{o}+p_{1j}\v{d_1}+p_{2j}\v{d_2}+\dots+p_{nj}\v{d_n}
$$
for $j=n+2,\dots,k$. Recall that we also have $\v{a_1}=\v{o}$.
\qed

\vskip.3cm

We illustrate the method by some examples.
In the examples we use the same norms as in Section \ref{1d}, but $\Delta$ is no longer equal to $\left(\prod_{i=1}^n |\v{d_i}| \right)^{1/n}$, but it still equals the absolute value of the $n$-th root of the lattice determinant of the computed vectors $\v{d}$. From the construction it follows that $|{\rm det}(\v{d_1},\dots,\v{d_n})|\cdot|\det(WQ)|=1$. Hence $\Delta = (|\text{det}(WQ)|)^{-1/n}$.

In our first example we give a detailed description of our method.

\vskip.3cm

\noindent {\bf Example 4.1.} We work with the same values as in Example 3.1. We choose $n=2$, $k=6$ and
$$
A=\{\v{a_1}=(0.814258,0),\ \v{a_2}=(1.294837,\sqrt{3}),\ \v{a_3}=(2.237840,\sqrt{5}),
$$
$$
\v{a_4}=(2.764132,\sqrt{7}),\ \v{a_5}=(4.295116,\sqrt{11}),\ \v{a_6}=(7.733842,\sqrt{13})\}.
$$
We choose an affine transformation where $\v{a_1}$ goes to $\v{a_1^N}=(0,0)$, $\v{a_6}$ goes to $\v{a_6^N}=(1,0)$ (these points yield the diameter), and $\v{a_4}$ (the point furthest from the line through $\v{a_1}$ and $\v{a_6}$) goes to $\v{a_4^N}=(0,1)$. Then we get $\v{a_2^N}=(-0.1867, 0.9091)$, $\v{a_3^N}=(-0.0526,0.9167)$, $\v{a_5^N}=(0.2432,0.9222)$. In fact, this transformation is given by $W(\v{a}-\v{o})$, with $\v{o}=\v{a_1}$ and
$$
W=
\begin{pmatrix}
0.2346&-0.1729\\
-0.3197&0.6136
\end{pmatrix}
.
$$
Choosing $\varepsilon=10^{-3}$, we apply the LLL algorithm to the matrix
$$
T=
\begin{pmatrix}
1&0&0&0&0\\
0&1&0&0&0\\
0&0&1&0&0\\
-0.1867 & -0.0526 & 0.2432 & 0.001 & 0\\
0.9091 & 0.9167 & 0.9222 &0 & 0.001
\end{pmatrix}
.
$$
Observe that we spared the columns corresponding to the vectors $\v{a_1^N}$, $\v{a_6^N}$ and $\v{a_4^N}$, hence $T$ is of type $(k-1)\times(k-1)=5\times 5$. We get the matrix
$$
B=
\begin{pmatrix}
-0.0280&0.0114&0.0182&-0.0130&-0.0280\\
-0.0284&0.0246&-0.0319&0.0060&0.0320\\
0.0275&-0.0358&-0.0023&-0.0610&0.0150\\
-0.0279&-0.0193&-0.0616&0.0240&-0.0270\\
-0.0345&-0.0202&0.02943&0.0260&0.0680
\end{pmatrix}
.
$$
Subsequently we compute the integer matrix
$$
S:= BT^{-1}=
\begin{pmatrix}
23&25&29&-13&-28\\
-28&-29&-31&6&32\\
-25&-17&1&-61&15\\
29&26&19&24&-27\\
-57&-61&-69&26&68
\end{pmatrix}
.
$$
We take the $2\times 2$ matrix
$$
Q=
\begin{pmatrix}
-13&-28\\
6&32
\end{pmatrix}
$$
in the right upper corner of $S$. Then a basis of a well approximating lattice to the points to $\v{a^N_1},\dots, \v{a^N_6}$ is given by the column vectors of the matrix
$$
-Q^{-1}=
\begin{pmatrix}
0.1290&0.1129\\
-0.0242&-0.0524
\end{pmatrix}
.
$$
That is, we have
$$
\v{d_1'}=(0.1290,-0.0242),\ \ \ \v{d_2'}=(0.1129,-0.0524).
$$
Recalling our choice for using the indices $1,6,4$ at the normalization step, as approximating points to $\v{a^N_1}, \dots, \v{a^N_6}$ we obtain
$$
\begin{pmatrix}
0&0\\
23&-28\\
25&-29\\
28&-32\\
29&-31\\
13&-6
\end{pmatrix}
\begin{pmatrix}
\v{d_1'}\\
\v{d_2'}
\end{pmatrix}
=
\begin{pmatrix}
0&0\\
-0.1935&0.9113\\
-0.0488&0.9153\\
0&1\\
0.2419&0.9234\\
1&0
\end{pmatrix}
.
$$
Here the entries in the second, third and fifth rows of the matrix on the left hand side come from the first three entries of the first two rows of $S$ (since these are the values $p_{ij}$ for $i=1,2$ and $j=1,2,3$). The zeroes in the first row are clear, since $\v{a_1^N}=(0,0)$. Finally, the entries in the sixth and fourth line come from the identity
$$
(-Q)\cdot (-Q^{-1})=
\begin{pmatrix}
1&0\\
0&1
\end{pmatrix}
=(\v{a_6^N},\v{a_4^N})
.
$$

We also get that besides $\v{o}=\v{a_1}$, a basis of a well approximating lattice to the original points $\v{a_1},\dots,\v{a_6}$ is given by the column vectors of
$$
-W^{-1}Q^{-1}=
\begin{pmatrix}
0.8457&0.6790\\
0.4012&0.2684
\end{pmatrix},
$$
that is, by
$$
\v{d_1}=(0.8457,0.4012),\ \ \ \v{d_2}=(0.6790,0.2684).
$$
Further, we get the following approximating points to $\v{a_1},\dots,\v{a_6}$:
$$
(0.8143,0), (1.2519,1.7132), (2.2642,2.2473),
$$
$$
(2.7641,2.6456),(4.2888,3.3154), (7.7338,3.6056).
$$
Using that $\Delta = 0.2132$, ${\rm diam}(A) = 7.8026$ we obtain
$$
N_{\Lambda,\v{o}}(A)=2.4244,\ \ \ N^{(2)}_{\Lambda,\v{o}}(A) =2.8598.
$$
These values may be compared with the ones $9.3622$ and $11.0453$, respectively, from Example 3.1. The greater flexibility leads to better upper bounds.

The corresponding values for $\varepsilon = 10^{-2}$ in place of $10^{-3}$ are $\v{d_1}=(0.9885,0.5151)$, $\v{d_2}=(-5.9039,-4.7061)$, $\Delta =1.2693$, diam$(A) = 7.8026$,
$$
N_{\Lambda,\v{o}}(A)=1.7633,\ \ \ N^{(2)}_{\Lambda,\v{o}}(A) =2.8511.
$$

\vskip.4cm

\noindent {\bf Example 4.2.} We work with the same values as in Example 3.2. We choose $n=2$, $k=6$, $\varepsilon = 10^{-3}$ and
$$
A=\{\v{a_1}=(0.814258,\sqrt{5}),\ \v{a_2}=(1.294837,0),\ \v{a_3}=(2.237840,\sqrt{13}),
$$
$$
\v{a_4}=(2.764132,\sqrt{3}),\ \v{a_5}=(4.295116,\sqrt{11}),\ \v{a_6}=(7.733842,\sqrt{7})\}.
$$
We take the affine transformation where $\v{a_2}$ goes to $\v{a_2^N}=(0,0)$, $\v{a_6}$ goes to $\v{a_6^N}=(1,0)$ (these points give the diameter), and $\v{a_3}$ (the furthest point from the line through $\v{a_2}$ and $\v{a_6}$) goes to $\v{a_3^N}=(0,1)$. The transformation is given by $W(\v{a}-\v{o})$ with $\v{o}=\v{a_2}$ and
$$
W=
\begin{pmatrix}
0.1740&-0.0455\\
-0.1277&0.3107
\end{pmatrix}
.
$$
We get $\v{a_1^N}=(-0.1854,0.7562)$, $\v{a_4^N} = (0.1768,0.3506)$, and $\v{a_5^N} = (0.3711,0.6475)$.
We apply the LLL algorithm with $\varepsilon = 10^{-3}$ to the matrix
$$
T=
\begin{pmatrix}
1&0&0&0&0\\
0&1&0&0&0\\
0&0&1&0&0\\
-0.1854 & 0.1768 & 0.3711 & 0.001&0\\
0.7562&0.3506&0.6475&0 & 0.001
\end{pmatrix}
$$
and obtain
$$
\begin{pmatrix}
-14&-7&-13&2&19\\
11&-3&-7&31&-7
\end{pmatrix}
$$
as the first two rows of $S$. By inverting the matrix
$$
Q
=
\begin{pmatrix}
2&19\\
31&-7
\end{pmatrix}
,
$$
and using the inverse of $W$ we get as a basis
$$
\v{d_1}=(-0.1232,-0.2161),\ \ \ \v{d_2}=(-0.1998,-0.0714)
$$
of an approximating lattice to the original points. This results in the following approximating points to $\v{a_1}, \dots,\v{a_6}$:
$$
(0.8227, 2.2396), (1.2948, 0), (2.2378, 3.6055),
$$
$$
(2.7567, 1.7267), (4.2951, 3.3088), (7.7338, 2.6458).
$$
The errors of approximation of $A$ are given by
$$
N_{\Lambda,\v{o}}(A)=1.3041,\ \ \ N^{(2)}_{\Lambda,\v{o}}(A) =1.5720.
$$
These values may be compared with the ones $8.6761$ and $10.2364$, respectively, from Example 3.2.

To see the influence of the choice of the $\varepsilon$ we have used our program to compare the results for $\varepsilon = 10^{-i}$ with $i=2,3, \dots, 10$. In Table \ref{tab1}, we give a summary of the results. Notice that, the lattices become smaller, but that the norms do not vary too much. The calculations were made in 20-digit precision.

\begin{table}[tb]
\begin{tabular}{|c|c|c|c|}
\hline
$\varepsilon$ & $\v{d_1}$ & $\v{d_2}$ & $N^{(2)}(A)$\\
\hline\hline
$10^{-2}$ & $(0.1862, 2.7095)$ & $(0.9430, 3.6056)$ & $1.1066$ \\
\hline
$10^{-3}$ & $(1.232, 2.161)\cdot 10^{-1}$ & $(1.998, 0.714) \cdot 10^{-1}$ & $1.5720$\\
\hline
$10^{-4}$ & $(8.800, 3.458) \cdot 10^{-2}$ & $(0.741, 6.068) \cdot 10^{-2}$ & $0.7874$ \\
\hline
$10^{-5}$ & $(3.402, 7.279) \cdot 10^{-1}$ & $(2.293, 5.052) \cdot 10^{-1}$ & $2.1818$ \\
\hline
$10^{-6}$ & $(2.122, -1.098) \cdot 10^{-3}$ & $(2.534, 1.508) \cdot 10^{-3}$ & $0.9039$ \\
\hline
$10^{-7}$ & $(9.513, 7.310) \cdot 10^{-4}$ & $(4.297, -0.853) \cdot 10^{-4}$ & $0.7786$ \\
\hline
$10^{-8}$ & $(4.539, 2.942) \cdot 10^{-4}$ & $(2.158, 2.222) \cdot 10^{-4}$ & $1.5469$ \\
\hline
$10^{-9}$ & $(4.239, 4.273) \cdot 10^{-5}$ & $(3.382, 0.792) \cdot 10^{-5}$ & $1.0819$ \\
\hline
$10^{-10}$ & $(1.027, 0.675)\cdot 10^{-5}$ & $(1.099, 0.311) \cdot 10^{-5}$ & $0.6925$\\
\hline
\end{tabular}\\[2pt]
\caption{Basis vectors and $N^{(2)}(A)$ errors for approximating lattices for $A$ for different values of $\varepsilon$.} \label{tab1}
\end{table}

One of the aims of the project is to recognize hidden structures. In the following example we started with linear combinations with integer coefficients of $(\lg 3, \lg 7)$ and $(\lg 5, \lg 8)$ (where $\lg x$ is the logarithm of $x>0$ to base $10$) and wondered whether the algorithm finds the underlying lattice. When we gave the set $A$ with 4-digit accuracy it did not,  but it did with 7-digit accuracy. With $\varepsilon = 10^{-4}$ we obtained the following.
\vskip.3cm

\noindent {\bf Example 4.3.}
We choose $n=2$, $k=6$, $\varepsilon = 10^{-4}$ and
$$
A=\{\v{a_1}=(0,0),\ \v{a_2}=(72.6836917, 103.2838586),
$$
$$
\v{a_3}=(41.2087354,66.9615022),\ \v{a_4}=(44.7461978, 62.8435663),
$$
$$
\v{a_5}=(51.1493167, 78.2045256),\ \v{a_6}=(10.8279763,11.4749913)\}.
$$
We choose an affine transformation where $\v{a_1}$ goes to $\v{a_1^N}=(0,0)$, $\v{a_2}$ goes to $\v{a_2^N}=(1,0)$, $\v{a_3}$ goes to $\v{a_3^N}=(0,1)$. The transformation is given by $W(\v{a}-\v{o})$ with $\v{o}=\v{a_1}=(0,0)$ and
$$
W=
\begin{pmatrix}
0.1096&-0.0675\\
-0.1691&0.1190
\end{pmatrix}
.
$$
We get $\v{a_4^N}=(0.6656, -0.0882)$, $\v{a_5^N} = (0.3312, 0.6569)$, and $\v{a_6^N} = (0.4129,-0.4655)$. We apply the LLL algorithm with $\varepsilon = 10^{-4}$ and find
$$
\begin{pmatrix}
27&-16&34&-35&42\\
22&53&-11&-41&-60
\end{pmatrix}
$$
as the first two rows of the basis transformation matrix $S$. Then with the usual process we obtain
$$
\v{d_1}=(0.6990,0.9031),\ \ \ \v{d_2} = (1.1761,1.7482)
$$
as a basis for an approximating lattice for the original points. Here we recognize the approximate values $\lg 5$, $\lg 8$, $\lg 15$ and $\lg 56$. This results in the following approximating points to $\v{a_1},\dots,\v{a_6}$:
$$
(0,0), (72.6837, 103.2839), (41.2087, 66.9615),
$$
$$
(44.7462, 62.8436), (51.1493, 78.2045), (10.8280, 11.4750).
$$
The errors of approximation of $A$ are given by
$$
N(A) \leq N_{\Lambda,\v{o}}(A)=1.067 \cdot 10^{-5},\ \ \ N^{(2)} \leq N^{(2)}_{\Lambda,\v{o}}(A) =1.721 \cdot 10^{-5}.
$$
These small errors indicate that the lattice $\Lambda$ is actually found. The coefficients derived from the first two rows of $S$ are given by
$$
(72.6837, 103.2839) = 35 (\lg 5, \lg 8) + 41 (\lg 15, \lg 56),
$$
$$
(41.2087, 66.9615)= -42 (\lg 5, \lg 8) + 60 (\lg 15, \lg 56),
$$
$$
(44.7462, 62.8436)= 27 (\lg 5, \lg 8) + 22 (\lg 15, \lg 56),
$$
$$
(51.1493, 78.2045)= -16 (\lg 5, \lg 8) + 53 (\lg 15, \lg 56),
$$
$$
(10.8280, 11.4750)= 34 (\lg 5, \lg 8)  -11 (\lg 15, \lg 56).
$$
This can be reduced to the values with which we have started:
$$
(72.6837, 103.2839) = 41 (\lg 3, \lg 7) + 76 (\lg 5, \lg 8),
$$
and so on.
\qed

\section{Fine-tuning $N_{o, \Lambda}^{(2)}(A)$}
\label{finetuning}

The results obtained in the previous section can be improved by applying the least squares algorithm in order to find the optimal values of $\v{o}$ and the lattice vectors for the values of the $q_i$'s and $p_i$'s selected after applying the LLL algorithm. As we have seen in Section \ref{nd}, the underlying idea is that an "origin" $\v{o}$ and a lattice spanned by $\v{d_1},\dots,\v{d_n}$ are chosen in such a way that the point $\v{a_j}$ is close to the lattice point
$$
\v{o}-q_{i_1(j-1)}\v{d_1}-q_{i_2(j-1)}\v{d_2}-\dots-q_{i_n(j-1)}\v{d_n}
$$
for $j=2,\dots,n+1$, and to
$$
\v{o}+p_{1j}\v{d_1}+p_{2j}\v{d_2}+\dots+p_{nj}\v{d_n}
$$
for $j=n+2,\dots,k$, respectively.

The least squares method enables us to optimize $\v{o}$, $\v{d_1},\dots,\v{d_n}$ with respect to the sum of the Euclidean distances between the points and the approximating lattice points. Notice that, when applying the least squares method, inversion of the matrix $Q$ as in Section \ref{nd} is no longer needed. We illustrate this by applying the Maple 15 procedure LeastSquares to some treated examples.

We stress that in fact this method minimizes the numerator of the main term of the norm, i.e. the expression
$$
\sqrt{\sum_{\v{a} \in A} |\v{a}-\v{o}-\Lambda|^2}.
$$
However, since the change in the basis vectors is minimal, we expect that the norm itself improves. This is supported by the examples below, too. Again, we compute in 10-digit precision, but write in 4-digit precision.

\vskip.3cm

\noindent {\bf Example 5.1.} This is a continuation of Example 2.2. We started from
$$
A=\{a_1=0,\ a_2=\sqrt{3}=1.7321,\ a_3=\sqrt{5}=2.2361,
$$
$$
a_4=\sqrt{7}=2.6458,\ a_5=\sqrt{11}=3.3166,\ a_6=\sqrt{13}= 3.6056\}
$$
and found in Example 2.2 the tuple
$$
(p_1,\dots,p_6)=(0,72,93,110,138,150).
$$
We apply the least squares algorithm and find $o=0.0007$, $d=0.0240$.
For this choice of $d$ and $o$ we have
$$
N^{(2)}_{d,o}(A)= 0.2766
$$
which is less than $0.3374$ found in Example 2.2.
\qed

\vskip.1cm

\noindent {\bf Example 5.2.} This is a continuation of Example 4.2. We started with
$$
A=\{\v{a_1}=(0.814258,\sqrt{5}),\ \v{a_2}=(1.294837,0),\ \v{a_3}=(2.237840,\sqrt{13}),
$$
$$
\v{a_4}=(2.764132,\sqrt{3}),\ \v{a_5}=(4.295116,\sqrt{11}),\ \v{a_6}=(7.733842,\sqrt{7})\}
$$
and obtained
$$
\begin{pmatrix}
-14&-7&-13&2&19\\
11&-3&-7&31&-7
\end{pmatrix}
$$
as the first two rows of $S$. We found as a basis of an approximating lattice
$$
\v{d_1}=(-0.1232,-0.2161),\ \ \ \v{d_2}=(-0.1998,-0.0714)
$$
and approximating points
$$
(0.8227, 2.2396), (1.2948, 0), (2.2378, 3.6055),
$$
$$
(2.7567, 1.7267), (4.2951, 3.3088), (7.7338, 2.6458)
$$
to the original $\v{a_1},\dots,\v{a_6}$. The error of approximation of $A$ was
$$
N^{(2)}_{\Lambda,\v{o}}(A)=1.5720.
$$
After applying the least squares algorithm to the matrices
$$
\begin{pmatrix}
1&-14&11\\
1&0&0\\
1&-19&7\\
1&-7&-3\\
1&-13&-7\\
1&-2&-31
\end{pmatrix},
\ \ \
\begin{pmatrix}
0.8227&2.2396\\
1.2948& 0 \\
2.2378&3.6055\\
2.7567&1.7267\\
4.2951&3.3088\\
7.7338& 2.6458
\end{pmatrix},
$$
we obtain $\v{o}=(1.2955, 0.0000)$ together with $\v{d_1}=(-0.1231,-0.2162)$, $\v{d_2}=(-0.1998,-0.0715)$ and get the following approximating points to $\v{a_1},\dots,\v{a_6}$:
$$
(0.8284, 2.2403), (1.2955, 0.0000), (2.2357, 3.6073),
$$
$$
(2.7567, 1.7280), (4.2947, 3.3112), (7.7365, 2.6492).
$$
The new error of approximation of $A$ is
$$
N^{(2)}_{\Lambda,\v{o}}(A) = 0.8302.
$$
This value may be compared with the value $1.5720$ from Example 4.2.
\qed

\end{document}